\documentclass [11pt] {article} 
\usepackage{graphicx}
\usepackage{amssymb}
\usepackage{amsmath}
\usepackage[section]{placeins}
\usepackage{mathrsfs}  
\usepackage{multirow}
\usepackage{amsthm}

\usepackage{subcaption}
\usepackage{mwe}
\usepackage[english]{babel}
\usepackage[scaled]{helvet}
\usepackage[customcolors,norndcorners]{hf-tikz}
\definecolor{couleurrouge}{RGB}{210,0,0}

\usepackage{mathtools}

\def\RR{I\kern-0.35em R\kern0.2em \kern-0.2em}
\def\BBN{I\kern-0.35em  N\kern0.2em \kern-0.2em}

\newcommand{\be}{\begin{equation}}
\newcommand{\ee}{\end{equation}}
\newcommand{\bea}{\begin{eqnarray}}
\newcommand{\eea}{\end{eqnarray}}
\newcommand{\beann}{\begin{eqnarray*}}
\newcommand{\eeann}{\end{eqnarray*}}

\newcommand{\R}{\mathbb{R}}

\newcommand{\ba}{\bar{a}}

\newcommand{\ta}{\tilde{a}}
\newcommand{\sign}{\text{sign}}
\newcommand{\ball}[2]{\overline{B_{#1}({#2})}}
\newcommand{\bydef}{\stackrel{\mbox{\tiny\textnormal{\raisebox{0ex}[0ex][0ex]{def}}}}{=}} 

\newtheorem{thm}{Theorem}[section]

\newtheorem{lem}[thm]{Lemma}

\DefineNamedColor{named}{BrickRed}      {cmyk}{0,0.89,0.94,0.22}

\begin{document}

\title{Traveling wave oscillatory patterns in a signed Kuramoto-Sivashinsky equation with absorption}

\author{
Yvonne Bronsard Alama \thanks{McGill University, Department of Mathematics and Statistics, 805 Sherbrooke Street West, Montreal, QC, H3A 0B9, Canada. {\tt yvonne.alama@mail.mcgill.ca}}
\and
Jean-Philippe Lessard\thanks{McGill University, Department of Mathematics and Statistics, 805 Sherbrooke Street West, Montreal, QC, H3A 0B9, Canada. {\tt jp.lessard@mcgill.ca}. This author was supported by NSERC.}
}

\date\today

\maketitle

\begin{abstract}
In this paper, a partial proof of a conjecture raised in \cite{MR2272794} concerning existence and global uniqueness of an asymptotically stable periodic orbit in a fourth-order piecewise linear ordinary differential equation is presented. The fourth-order equation comes from the study of traveling wave patterns in a signed Kuramoto-Sivashinsky equation with absorption. The proof is twofold. First, the problem of solving for the periodic orbit is transformed into a zero finding problem on $\R^4$, which is solved with a computer-assisted proof based on Newton's method and the contraction mapping theorem. Second, the rigorous bounds about the periodic orbit in phase space are combined with the theory of discontinuous dynamical systems to prove that the orbit is asymptotically stable.
\end{abstract}

\begin{center}
{\bf \small Key words.} 
{ \small Traveling wave patterns, Kuramoto-Sivashinsky model, discontinuous dynamical systems, periodic orbits, computer-assisted proofs}
\end{center}


\section{Introduction}

The Kuramoto-Sivashinsky equation 
\begin{equation} \label{eq:original_KS}
u_t + \nabla^4 u + \nabla^2 u -(\nabla^2 u)^2 = 0,
\end{equation}
where $\nabla^2$ is the Laplace operator and $\nabla^4$ is the biharmonic operator, is a fourth-order semilinear parabolic PDE which was originally introduced to model flame front propagation and later became a popular model to analyze weak turbulence or spatiotemporal chaos \cite{MR1430739,MR876914,MR1050912,KT,MR2496834,MR0502829}. In an attempt to study extinction phenomena, Galaktionov and Svirshchevskii consider in \cite{MR2272794} a modification of \eqref{eq:original_KS}, namely the {\em signed KS equation with absorption}
\begin{equation} \label{eq:signed_KS_with_absorption}
u_t + \nabla^4 u + \sign(u) - (\nabla^2 u)^2 = 0 .
\end{equation}
Considering equation \eqref{eq:signed_KS_with_absorption} on the real line (i.e. $u=u(\xi,t)$, with $\xi \in \R$ and $t \ge 0$), and following the approach of \cite{MR2272794}, we plug the traveling wave ansatz $u(\xi,t) = f(y)$ (with $y \bydef \xi - c t$) in \eqref{eq:signed_KS_with_absorption} which leads to the problem  
\[
- c f'(y) + f^{(4)}(y) + \sign f(y) - (f''(y))^2  = 0, \quad y \in (0,\infty) \text{ and } f(0)=0.
\]
Following the approach of \cite{MR2272794} (based on their experience studying the thin film equation), we only keep the highest derivative term in the equation and this yields 
\begin{equation} \label{eq:signed_simple_equation}
f^{(4)}(y) +  \sign f(y) = 0, \quad y \in (0,\infty) \text{ and } f(0)=0.
\end{equation} 
To capture the oscillatory component $\varphi(s)$ in the solution of \eqref{eq:signed_simple_equation}, we change coordinates $(y,f(y)) \mapsto (s,\varphi(s))$ via
\[
f(y) = y^4 \varphi (s), \quad \text{with } \quad s \bydef \ln y
\]
and plugging the transformation in \eqref{eq:signed_simple_equation} leads to the fourth-order piecewise linear ordinary differential equation
\begin{equation} \label{eq:fourth_order_ODE}  
\varphi^{(4)}(s) + 10 \varphi'''(s) + 35 \varphi''(s) + 50 \varphi'(s) + 24 \varphi(s) +\sign(\varphi(s)) = 0.
\end{equation}
This change of coordinates turns the search for a blowup solution into the search for a periodic solution. The purpose of the present paper is to give a partial proof of the Conjecture 3.2 on page 150 of \cite{MR2272794}, about solutions of equation [4], which we now state as a theorem.

\begin{thm} \label{thm:main_result}
Equation \eqref{eq:fourth_order_ODE} has a nontrivial asymptotically stable periodic solution.
\end{thm}

The periodic solution of Theorem~\ref{thm:main_result} is portrayed in Figure~\ref{fig:periodic_solution_profile} and the corresponding traveling wave pattern $u(\xi,t)=f(\xi-ct) = (\xi-ct)^4 \varphi (\ln(\xi-ct))$ is plotted in Figure~\ref{fig:tw_profile}. Note that we set $f(\xi-ct)=0$ for $\xi-ct \le 0$, and that we did not solve for the wave speed $c$.

\begin{figure}[h] 
\centering
\includegraphics[width=2.5in]{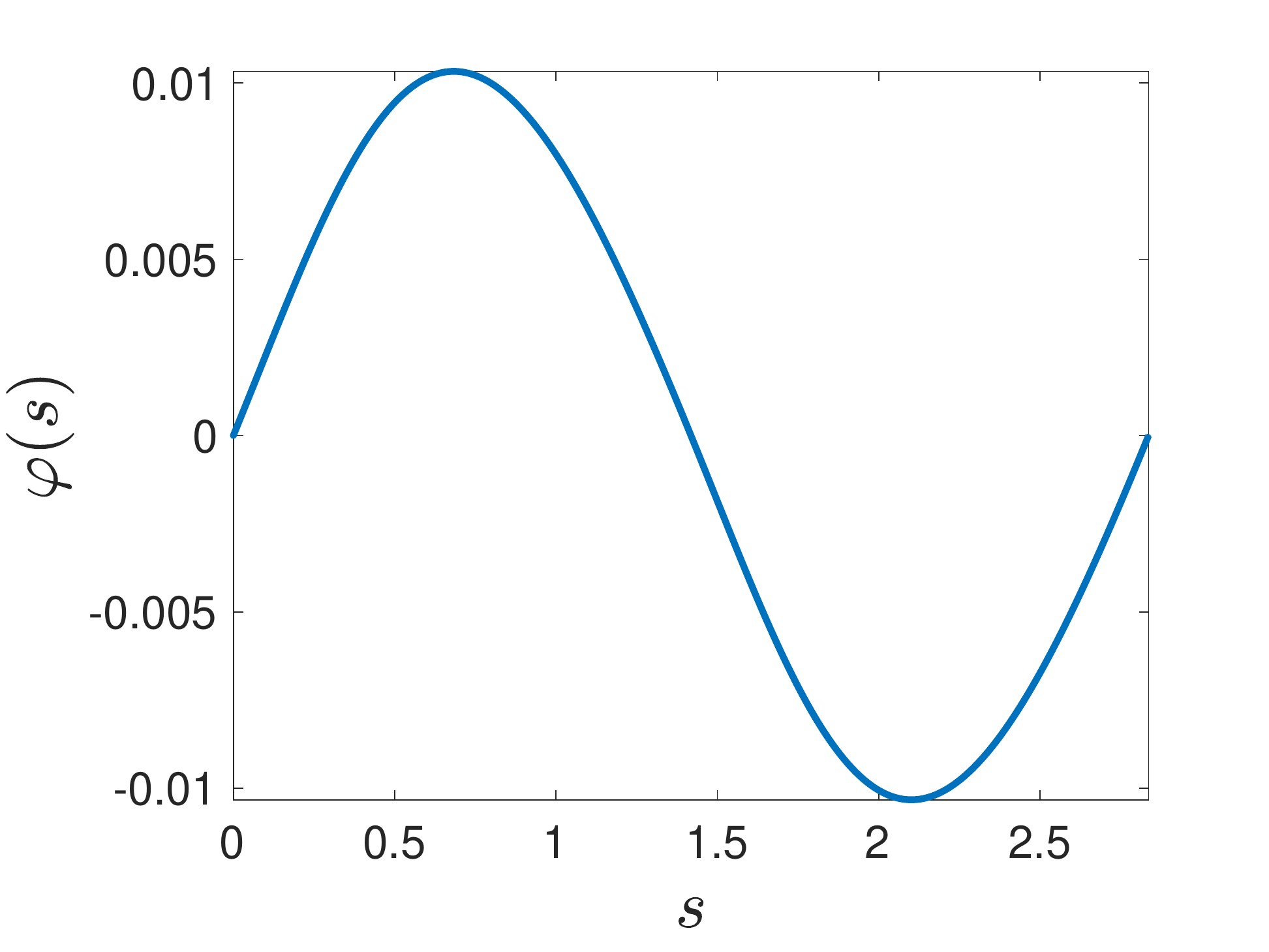}
~~~
\includegraphics[width=3in]{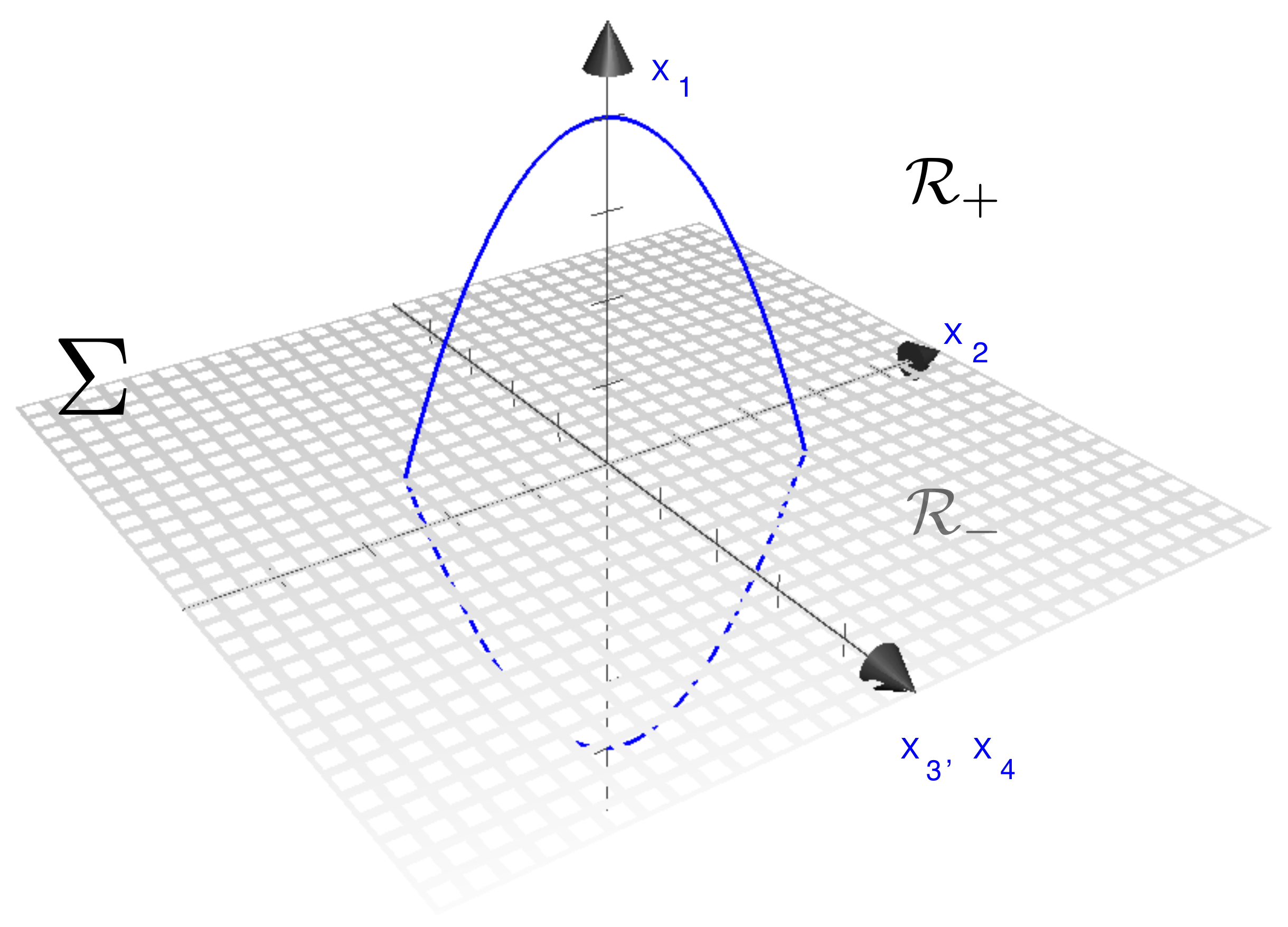}
\caption{Profile of the periodic solution of Theorem~\ref{thm:main_result} (left) and the corresponding orbit in the phase space (right).}
\label{fig:periodic_solution_profile}
\end{figure}

\begin{figure}[h] 
\centering
\includegraphics[width=6.5in]{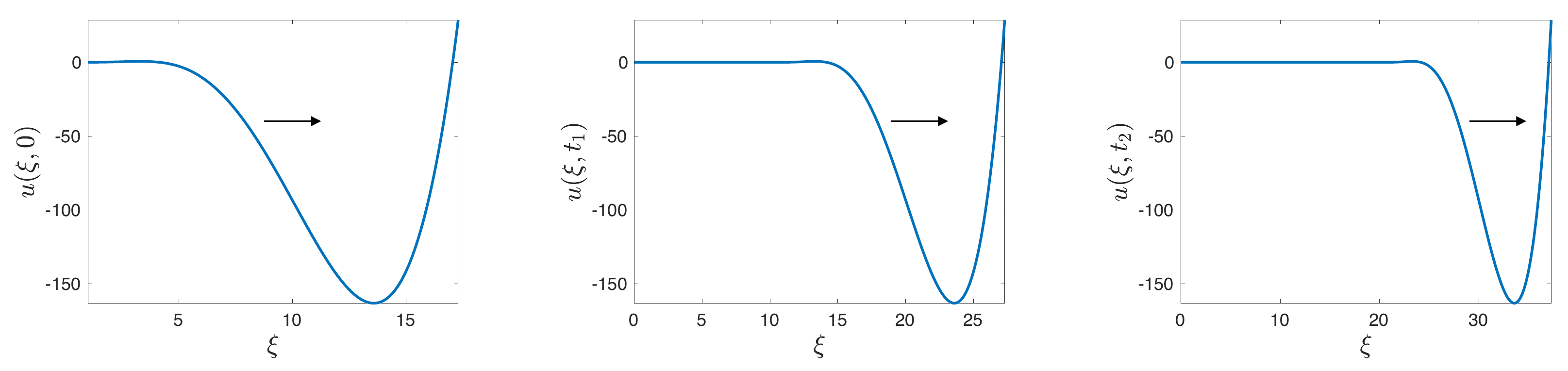}
\caption{Different snapshots of $u(\xi,t)=f(\xi-ct) = (\xi-ct)^4 \varphi (\ln(\xi-ct))$.}
\label{fig:tw_profile}
\end{figure}

The proof of Theorem~\ref{thm:main_result} has two parts. The first part of the proof (existence) is presented in Section~\ref{sec:existence}, where the problem of finding the periodic solution $\varphi(s)$ of \eqref{eq:fourth_order_ODE}  is transformed (via the symmetry argument of Lemma~\ref{lem:symmetry}) into a zero finding problem $F(a)=0$ where $F:\R^4 \to \R^4$ is defined in \eqref{eq:F_definition}. Proving the existence of $\ta \in \R^4$ such that $F(\ta)=0$ is done with a computer-assisted proof based on a Newton-Kantorovich type theorem (Theorem~\ref{thm:radPoly}). The second part of the proof is presented in Section~\ref{sec:stability}, where the rigorous enclosure of the periodic solution is combined with the theory of discontinuous dynamical systems to prove that the orbit is asymptotically stable. These two parts conclude the proof of Theorem~\ref{thm:main_result}.

\section{Existence: a computer-assisted proof} \label{sec:existence}

In this section, we prove the existence of a periodic solution $\varphi(s)$ of \eqref{eq:fourth_order_ODE}. To achieve this goal, we reformulate this into a zero finding problem $F(a)=0$ defined on $\R^4$. Proving the existence of a solution is done by verifying the hypotheses of Theorem~\ref{thm:radPoly} with the help of the digital computer and interval arithmetic (e.g. see \cite{MR0231516,Ru99a}). 

We begin by making the change of variables $x_1 \bydef \varphi$, $x_2 \bydef \varphi'$, $x_3 \bydef \varphi''$ and $x_4 \bydef \varphi'''$ to rewrite the fourth-order equation \eqref{eq:fourth_order_ODE} as the system
\begin{equation} \label{eq:system_ode}
\dot x = Mx +g(x) \bydef \begin{pmatrix}0&&1&&0&&0\\ 0 && 0 && 1 &&0 \\ 0&&0&&0&&1 \\ -24&&-50&&-35&&-10\end{pmatrix}\begin{pmatrix}x_1 \\ x_2 \\ x_3 \\ x_4\end{pmatrix} +\begin{pmatrix} 0 \\ 0 \\ 0 \\ -\sign(x_1)\end{pmatrix}. 
\end{equation}

Equation~\eqref{eq:system_ode} is a piecewise smooth dynamical system and changes rule as $x=(x_1,x_2,x_3,x_4)$ goes through the {\em switching manifold} defined by 
\[ 
\Sigma \bydef \{ x = (x_1,x_2,x_3,x_4) \in \R^4 : x_1 =0\}.
\]
The switching manifold $\Sigma$ separates the phase space $\R^4$ into the two regions $\mathcal R_+$ and $\mathcal R_-$ defined by $\mathcal R_+ \bydef \{ x \in \R^4 : x_1 \ge 0 \}$ and $\mathcal R_- \bydef \{ x \in \R^4 : x_1 \le 0 \}$. Denoting $b \bydef (0,0,0,-1)$, system \eqref{eq:system_ode} can then be written as 
\begin{equation} \label{eq:pws}
\dot x = 
\begin{cases}
f_{+}(x) \bydef Mx + b, & x \in \mathcal R_+
\\ 
f_{-}(x) \bydef Mx - b, & x \in \mathcal R_-.
\end{cases}
\end{equation}
Given $b_0 \in \{ \pm b\}$, the unique solution of $\dot x = Mx + b_0$, $x(0)=x_0 \in \R^4$ is given by
\begin{equation} \label{eq:variation_of_constant}
t \mapsto e^{Mt}x_0 + e^{Mt} \int_{0}^{t} e^{-Ms}b_0\, ds = e^{Mt} \left( x_0 + M^{-1} b_0 \right) - M^{-1}b_0.
\end{equation}
Note that 
\[
M^{-1} = 
\begin{pmatrix} -\frac{50}{24} & -\frac{35}{24} & -\frac{10}{24} & -\frac{1}{24} \\
1 & 0 & 0 & 0 
\\
0 & 1 & 0 & 0 
\\
0 & 0 & 1 & 0 
\end{pmatrix}
\]
Moreover, $M = P D P^{-1}$, $e^{Mt} = Pe^{Dt}P^{-1}$ where $e^{Dt} = {\rm diag} \left( e^{-4t} ~ e^{-3t} ~ e^{-2t} ~ e^{-t} \right)$ and 
\begin{equation}
P = \begin{pmatrix} 1& 1& 1&1\\ -4& -3& -2&-1\\16& 9&4&1\\-64 &-27&-8&-1\end{pmatrix}
\quad \text{and} \quad
P^{-1}= \begin{pmatrix}-1&-\frac{11}{6}&-1&-\frac{1}{6}\\ 4 & 7 & \frac{7}{2} & \frac{1}{2} \\ -6 & -\frac{19}{2} & -4 & -\frac{1}{2} \\ 4 & \frac{13}{3}& \frac{3}{2}&\frac{1}{6}\end{pmatrix}.
\end{equation}

We now introduce a result which exploits the symmetry of the problem and establishes a mechanism to obtain a periodic solution of \eqref{eq:pws}. 
\begin{lem} \label{lem:symmetry}
If there exist $L>0$ and a solution $\phi : [0,L] \to \R^4$ of $\dot x = f_{+}(x)$ with 
\begin{equation} \label{eq:continuity_condition}
\phi(L) = -\phi(0)
\end{equation}
and
\begin{equation} \label{eq:extra_condition}
\phi([0,L]) \subset \mathcal{R}_+,
\end{equation}
then $\phi(0),\phi(L) \in \Sigma$ and $\Gamma : [0, 2L] \to \R^4$ defined by
\begin{equation} \label{eq:Gamma_definition}
\Gamma(t) \bydef \begin{cases} 
\phi(t), & t \in [0,L] \\
-\phi(t - L), & t \in [L, 2L]
\end{cases}
\end{equation}
is a $2L$-periodic solution of \eqref{eq:pws}.
\end{lem}

\begin{proof}
First, $\phi([0,L]) \subset \mathcal{R}_+$ implies that $(\phi(0))_1,(\phi(L))_1 \ge 0$ and then $0 \le (\phi(L))_1 = -(\phi(0))_1 \le 0$. Hence, $(\phi(0))_1 = (\phi(L))_1 =0$, that is $\phi(0),\phi(L) \in \Sigma$.

Now, for $t \in [L,2L]$, $\psi(t)\bydef -\phi(t-L)$ solves $\dot x =f_{-}(x)$, as $\psi'(t) = - \phi'(t-L) = M(-\phi(t-L)) -b= M \psi(t) - b = f_-(\psi(t))$. Also $\phi([0,L]) \subset \mathcal{R}_+$ implies that $\psi([L,2L]) \subset \mathcal{R}_-$. Moreover, $\psi(L) = - \phi(0) = \phi(L) = -\psi(2L)$. 

By definition of $\Gamma(t)$ in \eqref{eq:Gamma_definition}, $\Gamma([0,L]) \subset \mathcal{R}_+$, $\Gamma([L,2L]) \subset \mathcal{R}_-$ and $\phi(L) = -\phi(L - L) = -\phi(0)$. Hence $\Gamma$ is continuous at $t=L$. 
 Finally, since $\Gamma(0) = \phi(0) = -\phi(L) = -\phi(2L - L) = \Gamma(2L)$, we conclude that $\Gamma(t)$ is a $2L$-periodic orbit of \eqref{eq:pws}. \qedhere
\end{proof}

To find the segment of the orbit $\phi : [0,L] \to \R^4$ solving $\dot x = f_{+}(x) = Mx + b$ as in Lemma~\ref{lem:symmetry}  we use formula \eqref{eq:variation_of_constant}, impose that the segment begins in the switching manifold (i.e. $\phi(0) \in \Sigma$) and that $\phi(L) = -\phi(0)$. Note that if $\phi(0) \in \Sigma$, then $\phi(0) = (0,a_2,a_3,a_4)$ for some $a_2,a_3,a_4 \in \R$. Using \eqref{eq:variation_of_constant}, the condition $\phi(L) = -\phi(0)$ reduces to solving
\begin{align*}
0 & = \phi(0)  + \phi(L) 
\\
& = \phi(0)  + e^{ML}\phi(0) + e^{ML} \int_{0}^{L} e^{-Ms}b  ~ ds
\\ 
& = P(I + e^{DL})P^{-1} \begin{pmatrix} 0 \\ a_2 \\ a_3 \\ a_4 \end{pmatrix} + \int_{0}^{L} Pe^{D(L-s)}P^{-1} b ~ ds.
\end{align*}
Denote $a \bydef (L,a_2,a_3,a_4)$ and let
\begin{equation} \label{eq:F_definition}
F(a) = 
F\begin{pmatrix}L \\ a_2 \\ a_3 \\ a_4\end{pmatrix}
\bydef P \left( e^{DL} + I \right) P^{-1}\begin{pmatrix}0 \\ a_2 \\ a_3 \\ a_4\end{pmatrix} + \int_{0}^{L} Pe^{D(L-s)}P^{-1}b\,ds.\end{equation}

To prove the existence of a periodic solution $\varphi(s)$ of \eqref{eq:fourth_order_ODE}, it is sufficient to prove the existence of a zero of $F:\R^4\to \R^4$ defined in \eqref{eq:F_definition} and then to verify the extra condition \eqref{eq:extra_condition}. While the rigorous verification of \eqref{eq:extra_condition} is done a-posteriori using interval arithmetics, the existence of a zero of $F$ is done using the {\em radii polynomial approach} (e.g. see \cite{MR2338393,MR3454370,MR3612178}) which is essentially the Newton-Kantorovich theorem (e.g. see \cite{MR0231218}). We now introduce this approach for a general $C^2$ map defined on $\R^n$. Endow $\R^n$ with the supremum norm $\|a\|_\infty = \max_{i=1,\dots,n}  |a_i|$ and denote by $\ball{r}{b} \bydef \left\{a \in \R^n  \mid  \| a - b \|_\infty \leq r \right\} \subset \R^n$ the closed ball of radius $r$ and centered at $b$. 

\begin{thm} \label{thm:radPoly}
Let $F\colon \R^n \to \R^n$ be a $C^2$ map. Consider $\ba \in \R^n$ (typically a numerical approximation with $F(\ba) \approx 0$). Assume that the Jacobian matrix $DF(\ba)$ is invertible and let $A \bydef DF(\ba)^{-1}$. Let $Y_0 \ge 0$ be any number satisfying
\begin{equation} \label{eq:Y0}
\| A F(\ba) \|_\infty \le Y_0.
\end{equation}
Given a positive radius $r_*>0$, let $Z_2=Z_2(r_*)$ be any number satisfying
\begin{equation} \label{eq:Z2_definition}
\sup_{a \in \ball{r_*}{\ba}} \biggl( \max_{1\leq i \leq n} \sum_{1\leq k,m \leq n} \Bigl| \sum_{1 \leq j \leq n} A_{ij} D^2_{km} F_j\bigl( a \bigr) \Bigr| \biggr) \le Z_2.
\end{equation}
Define the radii polynomial by
\begin{equation} \label{eq:radPoly}
p(r) \bydef Z_2 r^2 - r + Y_0.
\end{equation}
If there exists $r_0 \in (0,r_*]$ with $p(r_0) < 0$, then there exists a unique $\ta \in \ball{r_0}{\ba}$ such that $F(\ta) = 0$.
\end{thm}

\begin{proof}
Let $r \le r_*$ and consider $c \in \ball{r}{\ba}$. Applying the Mean Value Inequality and using \eqref{eq:Z2_definition},
\begin{equation} \label{eq:Z2_inequality}
||| A (DF(c) - DF(\ba)) |||_\infty  \le \sup_{\|h\|_\infty = 1} \sup_{ a \in \ball{r}{\ba}}  |||  A D^2F(a) h |||_\infty \| w - \ba\|_\infty
\le Z_2 r,
\end{equation}
where $||| \cdot |||_\infty $ denotes matrix norm.
Define the {\em Newton-like operator} $T \colon \R^n \to \R^n$ by $T(a) = a - AF(a)$.
Since $A$ is invertible, $F(\ta) = 0$ if and only if $T(\ta) = \ta$. 
Let $r_0 > 0$ be such that $p(r_0) < 0$. Hence $Z_2 r_0^2 + Y_0 < r_0$ and $Z_2 r_0 +  \frac{Y_0}{r_0} < 1$. Since $Y_0,Z_2 \ge 0$, one gets that
\begin{equation}\label{eq:polyBound2}
Z_2 r_0 < 1.
\end{equation}
For any $a \in \overline{B_{r_0}(\ba)}$, apply \eqref{eq:Z2_inequality} to get 
\[
||| DT(a) |||_{\infty} = ||| I - A DF(a) |||_{\infty} = ||| A [ DF(\ba) - DF(a)] |||_{\infty} \le  Z_2 r_0 < 1.
\]
Hence,
\begin{align*}
\|T(a) - \ba\|_\infty &\leq \|T(a) - T(\ba)\|_\infty + \|T(\ba) - \ba\|_\infty \\
&\leq \sup_{c \in \overline{B_{r_0}(\ba)}} ||| DT(c) |||_\infty \|a - \ba\|_\infty + \|A F(\ba) \|_\infty \\
&\leq \left(Z_2 r_0\right)r_0 + Y_0 < r_0.
\end{align*}
Then $T$ maps $\overline{B_{r_0}(\ba)}$ into itself. Finally, given $a_1, a_2 \in \overline{B_{r_0}(\ba)}$ combine 
 \eqref{eq:polyBound2} with the Mean Value Inequality to get
\[
\| T(a_1) - T(a_2) \|_\infty \le \sup_{c \in \overline{B_{r_0}(\ba)}} ||| DT(c) |||_\infty \| a_1 - a_2 \|_\infty 
\le (Z_2 r_0)  \|a_1 - a_2\|_\infty 
\le \kappa \|a_1 - a_2 \|_\infty,
\]
where $\kappa \bydef Z_2 r_0 < 1$.  Then, by the Contraction Mapping Theorem,
$T$ has a unique fixed point $\ta \in \overline{B_{r_0}(\ba)}$.  It follows from the invertibility of $A$ that $\ta$
is the unique zero of $F$ in $\overline{B_{r_0}(\ba)}$.  \qedhere
\end{proof}

We now apply Theorem~\ref{thm:radPoly} to prove the existence of a zero of $F$ defined in \eqref{eq:F_definition}. This begins by computing an approximate solution. Applying Newton's method, we find an approximate zero of $F$ given by
\begin{equation}
\ba = 
\begin{pmatrix}
1.418316134968973 \\ 
2.245235091886104 \times10^{-2} \\ 
8.358590910573891 \times10^{-3} \\ 
-4.883983455701284\times10^{-2}
\end{pmatrix}. 
\end{equation}

Then, using INTLAB (see \cite{Ru99a}) we compute rigorous enclosures of $DF(\ba)$ and $A \bydef DF(\ba)^{-1}$. We then verify rigorously that $Y_0 \bydef 7.4 \times 10^{-15} \geqslant \|{AF(\ba)}\|_{\infty}$, which settles the computation of the bound \eqref{eq:Y0}. 

The next bound to compute is $Z_2$ satisfying \eqref{eq:Z2_definition}.  
The only non zero second partial derivatives are the terms $\frac{\partial^2 F_j}{\partial a_1 \partial a_k}$ for $j, k \in \{1,\dots,4\}$, where we note that by Clairaut's theorem $\frac{\partial^2 F_j}{\partial a_1 \partial a_k} = \frac{\partial^2 F_j}{\partial a_k \partial a_1}$ for $k\in \{2,3,4\}, j \in \{1,\dots,4\}$. Hence, we can write the bound \eqref{eq:Z2_definition} as
\begin{align} \nonumber
Z_2(r_*)& \ge \max_{1\le i\le 4} \left\{ \sup_{c \in {\overline {B_{r_*}(\ba)}}} 
\left(
\left| \sum_{1\le j\le 4} A_{i,j} {\partial^2 F_j\over{\partial^2 a_1}} (c) \right| +2\left| \sum_{1\le j\le 4} A_{i,j} {\partial^2 F_j\over{\partial a_1\partial a_2}} (c) \right| \right. \right.
\\ 
&\hspace{3.5cm} + \left. \left.  2 \left| \sum_{1\le j\le 4} A_{i,j} {\partial^2 F_j\over{\partial a_1\partial a_3}} (c) \right| 
+2 \left| \sum_{1\le j\le 4} A_{i,j} {\partial^2 F_j\over{\partial a_1\partial a_4}} (c) \right| \right) \right\}.
\label{eq:Z2_our_problem}
\end{align}
Choosing $r_* = 0.01$ we use interval arithmetic to obtain that $Z_2 \bydef 41$ satisfies \eqref{eq:Z2_our_problem}.

Therefore, for $r\leqslant r_*$ the radii polynomial is given by
\[
p(r) = 41 r^2 -r +  7.4 \times 10^{-15}.
\]
Using interval arithmetic, we show that for every $r_0 \in [7.5 \times 10^{-15}, 0.01]$, $p(r_0)<0$.
By Theorem~\ref{thm:radPoly}, there exists a unique zero $\ta$ of $F$ in $\ball{7.5 \times 10^{-15}}{\ba}$. Denote $\ta = (\tilde L,\ta_2,\ta_3,\ta_4)$. Then since $|\tilde L- \ba_1| \le \| \ta - \ba \|_\infty \le  7.5 \times 10^{-15}$ and $\ba_1 = 1.418316134968973$, we conclude that $\tilde L>0$. By construction, 
\begin{equation} \label{eq:tilde_phi}
\tilde \phi(t) \bydef e^{Mt} \left( \tilde \phi(0) + M^{-1} b \right) - M^{-1}b
= e^{Mt} \left( \begin{pmatrix} 0 \\ \ta_2 \\ \ta_3 \\ \ta_4 \end{pmatrix} + M^{-1} b \right) - M^{-1}b
\end{equation}
defines a solution $\tilde \phi : [0,\tilde L] \to \R^4$ of $\dot x = f_{+}(x)$ with $\tilde \phi(\tilde L) = -\tilde \phi(0)$. The last hypothesis which needs to be verified to apply Lemma~\ref{lem:symmetry} is the condition \eqref{eq:extra_condition}, that is $\tilde \phi([0,\tilde L]) \subset \mathcal{R}_+$.
Using a MATLAB program using INTLAB, we consider a uniform time mesh (of size $300$) of the time interval $[0,\tilde L]$, that is $0=t_0<t_1< \cdots < t_{300} = \tilde L$. For each mesh interval $I_k=[t_{k-1},t_k]$ ($k=1,\dots,300$), the code computes an interval enclosure of $\tilde \phi(I_k)$ using formula \eqref{eq:tilde_phi}. Then the code verifies that $\tilde \phi_1(t)>0$ for all $t \in I_k$ and $k = k_1,\dots,k_2$ for some $1 < k_1 < k_2 < 300$. This implies that $\tilde \phi([t_{k_1-1},t_{k_2}]) \subset \mathcal R_+$. Afterward, it verifies that $\tilde \phi'_1(t) = \tilde \phi_2(t)>0$ for all $t \in I_k$ and $k = 1,\dots,k_1-1$. Hence $\tilde \phi_1(t)$ is strictly increasing over the interval $[0,t_{k_1-1}]$, and since $\tilde \phi_1(0)=0$, it follows that $\tilde \phi_1(t)>0$ for all $t \in (0,t_{k_1-1}]$, that is $\tilde \phi([0,t_{k_1-1}]) \subset \mathcal R_+$. Similarly, the code verifies that $\tilde \phi'_1(t) = \tilde \phi_2(t)<0$ for all $t \in I_k$ and $k = k_2+1,\dots,300$. Hence $\tilde \phi_1(t)$ is strictly decreasing over the interval $[t_{k_2},\tilde L]$, and since $\tilde \phi_1(\tilde L)=0$, it follows that $\tilde \phi_1(t)>0$ for all $t \in [t_{k_2},\tilde L)$, that is $\tilde \phi([t_{k_2},\tilde L]) \subset \mathcal R_+$. We conclude that 
\[
\tilde \phi([0,\tilde L]) = 
\tilde \phi([0,t_{k_1-1}])
\cup
\tilde \phi([t_{k_1-1},t_{k_2}])
\cup
\tilde \phi([t_{k_2},\tilde L])
\subset \mathcal{R}_+.
\]
Hence $\tilde \phi : [0,\tilde L] \to \R^4$ verifies the hypotheses of Lemma~\ref{lem:symmetry}. We conclude that $\tilde \phi(0),\tilde \phi(\tilde L) \in \Sigma$ and that
\begin{equation} \label{eq:tilde_Gamma_definition}
\tilde \Gamma(t) \bydef \begin{cases} 
\tilde \phi(t), & t \in [0,\tilde L] \\
-\tilde \phi(t - \tilde L), & t \in [\tilde L, 2\tilde L]
\end{cases}
\end{equation}
is a $2\tilde L$-periodic solution of \eqref{eq:pws}. All the computational steps described in this section are carried out in the MATLAB program {\tt Proof.m} available at \cite{webpage}.

\section{Asymptotic stability} \label{sec:stability}

In this section, we demonstrate that the $2\tilde L$-periodic orbit $\tilde \Gamma(t)$ defined in \eqref{eq:tilde_Gamma_definition} is asymptotically stable using the theory of discontinuous dynamical systems (e.g. see \cite{MR2368310}). We do this by computing the monodromy matrix $X(2\tilde L)$ of $\tilde \Gamma$ and show that all its nontrivial Floquet multipliers have modulus less than one.

Define $h:\R^4 \to \R$ by $h(x) \bydef x_1$ so that the switching manifold is given by
\[
\Sigma = \{ x \in \R^4 : h(x) = 0 \}.
\]
Denote by $\ta^{(1)} \bydef \tilde \phi(\tilde L)=(0,-\ta_2,-\ta_3,-\ta_4)^T$ the point at which $\tilde \Gamma$ crosses $\Sigma$ coming from $\mathcal R_+$ and entering in $\mathcal R_-$. Similarly, denote by $\ta^{(2)} \bydef (0,\ta_2,\ta_3,\ta_4)^T$ the point at which $\tilde \Gamma$ crosses $\Sigma$ coming from $\mathcal R_-$ and entering in $\mathcal R_+$. For that reason, $\tilde \Gamma$ is called a {\em crossing periodic orbit} (e.g. see \cite{MR3285646,MR2012652}).
Denote by $\tilde T \bydef 2\tilde L$ the period of $\tilde \Gamma$, $\bar{t} \bydef \tilde L$, and denote by $\Phi(t,x_0)$ the solution of \eqref{eq:pws} at time $t$ with initial condition $x_0$. 
Then (i.e. see \cite{MR3662364}) the monodromy matrix is given by
\begin{equation} 
X(\tilde T) = X(\tilde T, \bar{t})S_{-+}(\ta^{(2)})X(\bar{t}, 0)S_{+-}(\ta^{(1)}) 
\end{equation}
where 
\begin{align*}
S_{+-}(\ta^{(1)}) &\bydef I + \left( \frac{(f_{-}-f_{+})} {\nabla h^{T} \cdot f_{+}} \nabla h^{T} \right)  (\ta^{(1)})
\\
S_{-+}(\ta^{(2)}) &\bydef I + \left( \frac{(f_{+}-f_{-})} {\nabla h^{T} \cdot f_{-}} \nabla h^{T} \right) (\ta^{(2)})
\end{align*}
are called the {\em saltation matrices}, and where the fundamental matrix solutions $X(t,0)$ and $X(t,\bar t)$ satisfy 
\begin{align*}
\dot X(t, 0) &= Df_{-}(\Phi(t,\tilde{a}^{(1)}))X(t, 0) = M X(t,0), \quad \text{for } t \in [0,\tilde L], \quad \text{with }X(0, 0) = I
\\ 
\dot X(t, \tilde L) &= Df_{+}(\Phi(t,\tilde{a}^{(1)}))X(t, L) = M X(t,\tilde L), \quad \text{for } t \in [\tilde L,2\tilde L], \quad \text{with }X(\tilde L, \tilde L) = I.
\end{align*}
This implies that $X(t,0) = e^{Mt}$ and therefore $X(\bar{t},0)= X(\tilde L,0) = e^{M\tilde L}$. Similarly, $X(t,\bar{t}) = e^{M(t-\tilde L)}$ and then $X(\tilde T,\bar{t}) = X(2\tilde L,\tilde L) = e^{M\tilde L}$.

Since $\nabla h = (1,0,0,0)^T$ we obtain that $\nabla h^{T} \cdot f_{+}(\ta^{(1)}) =  -\ta_2$ and $\nabla h^{T} \cdot f_{-}(\ta^{(2)}) =  \ta_2$. Simple computations yield
\[
S_{+-}(\ta^{(1)}) = 
\begin{pmatrix} 
1 & 0 & 0 & 0 \\ 
0 & 1 & 0 & 0 \\
0 & 0 &1 & 0 \\ 
-\frac{2} {\ta_2} & 0 & 0 & 1 
\end{pmatrix} = S_{-+}(\ta^{(2)}).
\]

Hence, the monodromy matrix is given by
\begin{equation}  \label{eq:monodromy}
X(2 \tilde L) = e^{M\tilde L} 
\begin{pmatrix} 
1 & 0 & 0 & 0 \\ 
0 & 1 & 0 & 0 \\
0 & 0 &1 & 0 \\ 
-\frac{2}{\ta_2} & 0 & 0 & 1 
\end{pmatrix}
e^{M \tilde L} 
\begin{pmatrix} 
1 & 0 & 0 & 0 \\ 
0 & 1 & 0 & 0 \\
0 & 0 &1 & 0 \\ 
-\frac{2}{\ta_2} & 0 & 0 & 1 
\end{pmatrix}.
\end{equation} 

Using interval arithmetics and that $|\tilde L - \ba_1|,|\ta_2 - \ba_2| \le 7.5 \times 10^{-15}$, we compute rigorously an interval enclosure of \eqref{eq:monodromy} and using the rigorous computational method from \cite{MR3204427} we prove that the spectrum $\sigma(X(2 \tilde L))$ of $X(2 \tilde L)$ satisfies
\[
\sigma(X(2 \tilde L)) \subset \bigcup_{i=1}^4 B_i
\]
where
\begin{align*}
B_1 &\bydef [   0.99999989798820,   1.00000010201179] 
\\
B_2 &\bydef [   0.05862265751705,   0.05862286154064] 
\\
B_3 &\bydef [   0.00000059034712,   0.00000079437071] 
\\
B_4 &\bydef [   0.00001170839977,   0.00001191242336].
\end{align*}
This rigorous computation is carried out in the MATLAB program {\tt Proof.m} available at \cite{webpage}. From this, we conclude that three Floquet multipliers of $\tilde \Gamma$ have modulus strictly less than one. This concludes the proof that the $2\tilde L$-periodic orbit $\tilde \Gamma(t)$ defined in \eqref{eq:tilde_Gamma_definition} is asymptotically stable, and hence the proof of Theorem 1.1.


\end{document}